\documentclass[12pt]{article}
\usepackage{amsmath,amsxtra,amssymb,latexsym, amscd,amsthm}

\advance\voffset-1truecm\relax \advance\hoffset-1truecm\relax

scaled \magstep2  \font\tac=cmcsc10
scaled \magstep1  

\def\R{{\Bbb R}}
\def\N{{\Bbb N}}
\def\R{{\Bbb R}}
\def\C{{\Bbb C}}
\def\P{{\Bbb P}}
\def\M{{\Bbb M}}
\def\Z{{\Bbb Z}}

\numberwithin{equation}{section}

\newtheorem{lemma}{Lemma}[section]
\newtheorem{theorem}[lemma]{Theorem}

\newtheorem{proposition}[lemma]{Proposition}
\newtheorem{definition}[lemma]{Definition}
\newtheorem{remark}[lemma]{Remark}
\def\leq{\leqslant}

\begin{document}
\title{ AN EXTENSION OF THE CARTAN-NOCHKA  SECOND MAIN THEOREM
 FOR HYPERSURFACES}

\author{\tac  Gerd Dethloff, Tran Van Tan and Do Duc Thai}
\date{$\quad$}
\maketitle
\vspace{-0.5cm}
\begin{abstract}
\noindent In 1983,  Nochka proved a conjecture of Cartan on defects of holomorphic curves in $\C P^n$ relative to a possibly degenerate set of hyperplanes. In this paper, we generalize the Nochka's theorem to the case of curves in a complex projective variety intersecting hypersurfaces in subgeneral position.
\end{abstract}

\section{Introduction and statements}
Let $f$ be a holomorphic mapping of $\C$ into $\C P^M,$ with a
reduced representation $f=(f_0:\cdots:f_M).$ The characteristic
function $T_f(r)$ of $f$ is defined by
\begin{align*}
T_f(r):=\frac{1}{2\pi}\int_0^{2\pi}\log\Vert f(re^{i\theta})\Vert
d\theta,
\end{align*}
where $\Vert f\Vert:=\max\{|f_0|,\dots,|f_M|\}.$

\hbox to 5cm {\hrulefill }

 \noindent {\small Mathematics Subject
Classification 2000: Primary 32H30; Secondary 32H04, 32H25, 14J70.}

\noindent {\small Key words and phrases:  Nevanlinna theory, Second
Main Theorem.}

 \noindent{\small The second named author was partially supported by the  Abdus Salam International
Centre for Theoretical Physics.}

Let $L$ be a positive integer or $+\infty,$ and let $\nu$ be a
divisor on $\C.$ Set $\nu^{[L]}(z):=\min \{ \nu (z), L \}.$ The
truncated counting function to level $L$ of $\nu$ is defined  by
\begin{align*}
N^{[L]}_\nu(r) := \int\limits_1^r \frac{\sum_{\mid z\mid
<t}\nu^{[L]}(z)}{t}dt\quad (1 < r < +\infty).
\end{align*}

Let $\varphi$ be a nonzero meromorphic function on $\C$. Denote by
$\nu_\varphi$ be the zero divisor of $\varphi.$ Set
$N_\varphi^{[L]}(r):=N_{\nu_\varphi}^{[L]}(r).$

Let $D$ be a hypersurface in $\C P^M$ of degree $d\geq 1.$ Let
$Q\in\C[x_0,\dots,x_M]$ be a homogeneous polynomial of degree $d$
defining $D.$ Set
 $\nu_D^{[L]}(f):=\nu_{Q(f)}^{[L]},$ and $ N_f^{[L]}(r,
 D):=N_{Q(f)}^{[L]}(r).$
 For brevity we will
omit the character ${}^{[L]}$ in the counting function and in the
divisor if $L=+\infty.$

For the holomorphic function $\varphi,$ we have the following  Jensen's formula
\begin{align*}
N_\varphi(r)=\int_0^{2\pi}\log|\varphi(re^{i\theta})|\frac{d\theta}{2\pi}+O(1).
\end{align*}

Let $V \subset \C P^M$ be a smooth complex projective variety of
dimension $n\geq 1.$ Let $D_1,\dots,D_k$ $(k\geq 1)$ be hypersurfaces in 
$\C P^M$ of degree $d_ j.$
The hypersurfaces $D_1,\dots,D_k$   are said to be in {\it general position in $V$} if  for any distinct indices $1\leq i_1<\dots<i_s\leq k,\;(1\leq s\leq n+1),$ there exist hypersurfaces $D'_1,\dots,D'_{n+1-s}$ in $\C P^M$ such that $$V \cap D_{i_1}\cap\cdots\cap D_{i_s}\cap D'_1\cap\cdots\cap D'_{n+1-s}=\varnothing$$
(see Noguchi-Winkelmann \cite{NW} and Ru \cite{R3} for similar definitions).
In particular for hypersurfaces $D_1,\dots,D_k$ in general position in $V$, we have $V\not\subseteq D_j$ for all $j=1,...,k$. By definition, we also call an empty set of hypersurfaces
in  $\C P^M$ to be in general position in $V$.
\begin{definition}
 Let $N\geq n$ and $q\geq N+1.$
Hypersurfaces $D_1,\dots, D_q$ in $\C P^M$ with
$V\not\subseteq D_j$ for all $j=1,...,q$
are said to be in $N$-subgeneral position in $V$ if the two following conditions are satisfied:

 (i)$\quad$ For every $1\leq j_0<\cdots<j_N\leq q,$ $V \cap D_{j_0}\cap\cdots\cap D_{j_N}=\varnothing.$

 (ii)$\quad$ For any subset $J\subset\{1,\dots,q\}$ such that  $0<\#J\leq n+1$ and $\{D_j,\; j\in J\}$ are in general position in $V$
 and $V \cap (\cap_{j\in J}D_j)\not= \emptyset$, there exists an irreducible component $\sigma_J$ of $V \cap (\cap_{j\in J}D_j)$ with $\dim\sigma_J=\dim \big( V \cap (\cap_{j\in J}D_j)\big)$ such that for any $i\in\{1,\dots,q\}\setminus J$, if $\dim \big( V \cap (\cap_{j\in J}D_j)\big)= \dim \big(V \cap D_i \cap (\cap_{j\in J}D_j)\big),$ then $D_i$ contains $\sigma_J.$
\end{definition}
We first remark that if $V= \C P^M$ is a complex projective space and $\{D_j\}_{j=1}^q$ are hyperplanes, then the condition (ii) in the above definition is automatically satisfied. We also note that in the case where $N=n,$ the condition (i) implies the condition (ii). Therefore, in this case the above definition coincides with the concept of general position. \\

We finally construct an example of hypersurfaces in $(n+1)$-subgeneral position in $V$, which are, however, not in general position in $V$: Let $D_1,\dots, D_q$ $(q\geq n+1)$ be  hypersurfaces  in $\C P^M$ in general position in $V$. Let $\{J_1,\dots,J_K\}$ $(K=\binom{q}{n})$ be the set of all subsets $J$ of $\{1,\dots,q\}$ such that $\# J=n.$ It is clear that $0<\#\big( V \cap (\cap_{j\in J_i}D_j)\big)<\infty$ for all $1\leq i\leq K.$   We define hypersurfaces $D_{t_1},\dots,D_{t_K}$ in $\C P^m$ by induction as follows: Take a hypersurface  $D_{t_1}$ passing through a point $A_1\in V \cap (\cap_{j\in J_1}D_j),$ but not containing any irreducible component $\sigma$ of $V \cap (\cap_{j\in J}D_j)$ with $\dim\sigma=\dim\big( V \cap (\cap_{j\in J}D_j) \big)$ for all $J_1\ne J\subset \{1,\dots,q\}$ with $0<\# J\leq n$ (note that the number of these irreducible components $\sigma$ is finite, and  $\{A_1\}\ne\sigma,$ since $D_1,\dots,D_q$ are in general position in $V$). Then, for any $\varnothing\ne J\subset\{1,\dots,q,t_1\},\;\#J\leq n+1,$ $J\ne J_1\cup\{t_1\},$ the hypersurfaces $\{D_j,\;j\in J\}$ are in general position in $V$. Assume that hypersurfaces $D_{t_1},\dots,D_{t_{i-1}}$ $(2\leq i\leq K)$ in $\C P^M$ are chosen, we next  choose a hypersurface $D_{t_i}$  in $\C P^M$ passing through a point $A_i\in V \cap (\cap_{j\in J_i}D_j),$ but  not containing any irreducible component $\sigma$ of   $V \cap (\cap_{j\in J}D_j)$ with $\dim\sigma=\dim \big( V \cap (\cap_{j\in J}D_j)\big)$ for any $J_i\ne J\subset \{1,\dots,q,t_1,\dots,t_{i-1}\}$ with $0<\# J\leq n$  (note that $\{A_i\}\ne\sigma$ since $\{D_j,\;j\in J'\}$ are in general position in $V$ for all $ J'\subset\{1,\dots,q,t_1,\dots,t_{i-1}\},\;0<\#J'\leq n+1,\;J'\ne J_s\cup\{t_{s}\}\;(s=1,\dots,i-1)$). By our choices of the $D_{t_i}$'s, for any $J\subset\{1,\dots,q,t_{1},\dots,t_K\},\# J\leq n+1,$ the hypersurfaces $\{D_j,\;j\in J\}$ are in general position in $V$ except in the cases $J=J_i\cup\{t_i\}\;(i=1,\dots,K).$ Therefore for any $\varnothing \ne J\subset \{1,\dots,q,t_1,\dots,t_K\},\;\#J\leq n+1$ such that $\{D_j,\;j\in J\}$ are in general position in $V$ and for any $i\in\{1,\dots,q,t_1,\dots,t_K\}\setminus J,$ we have that $\dim \big(V \cap D_i \cap (\cap_{j\in J}D_j)\big)=\dim \big( V \cap (\cap_{j\in J} D_j)\big)$ if and only if either $\#J=n+1$ and then $V \cap (\cap_{j\in J}D_j)=\varnothing$ or there exists $k\in\{1,\dots,K\}$ such that $\{i\} \cup J=\{t_k\} \cup J_k$.  This implies that for $N=n+1$, the hypersurfaces $D_1,\dots, D_q,D_{t_1},\dots,D_{t_K}$ satisfy the conditions $(i)$  and $(ii)$ of Definition 1.1, and, hence,  they are in $(n+1)$-subgeneral position. But, they are not in general position.\\

 In 1933, Cartan \cite{C} proved the Second Main Theorem for linearly nondegenerate holomorphic mappings of $\C$ into $\C P^n$ intersecting hyperplanes in general position. He also proposed a conjecture for the case where the hyperplanes are only in subgeneral position. This conjecture was solved by Nochka \cite{N}.

\noindent As usual, by the notation ``$\Vert P$" we mean the assertion $P$
holds for all $r \in [1, +\infty)$ excluding a Borel subset $E$ of
$(1, +\infty)$ with $\displaystyle{\int\limits_E} dr < +  \infty$.
\begin{theorem}[Nochka]
Let $f$ be a linearly nondegenerate holomorphic mapping of $\C$ into $\C P^n$ and let $H_1,\dots,H_q$ be hyperplanes in $\C P^n$ in N-subgeneral position, where $N\geq n$ and  $q\geq 2N-n+1.$  Then, for every $\epsilon>0,$
\begin{align*}
\big\Vert(q-2N+n-1-\epsilon)T_f(r)\leq\sum_{j=1}^qN_f^{[n]}(r, H_j).
\end{align*}
\end{theorem}
Recently, the Second Main Theorem for the case of
hypersurfaces in general position was established  by Ru (\cite{R2}, \cite{R3}), see also Dethloff and Tan
\cite{DT}. For the case where hypersurfaces are not in general position, in \cite{ThTh} Thai and Thu  obtained a Second Main Theorem for algebraically non-degenerate holomorphic maps $f:\C \rightarrow \C P^k
\subset \C P^n$, without truncated multiplicities, and for a special class of hypersurfaces in $\C P^n$.

In 2009, Ru \cite{R3} proved that
\begin{theorem}\label{Th3}
Let $V\subset \C P^M$ be a smooth complex projective variety of
dimension $n\geq 1.$ Let $f$ be an algebraically nondegenerate
holomorphic mapping of $\C$ into $V.$ Let $D_1,\dots ,D_q $ be
hypersurfaces in $\C P^M$ of degree $d_j,$ in general position in
$V.$ Then for every $\epsilon>0,$
\begin{align*}
\Big\Vert(q-n-1-\epsilon)T_f(r)\leq\sum_{j=1}^q\frac{1}{d_j}N(r,D_j).
\end{align*}
\end{theorem}
Motivated by the case of hyperplanes, in this paper we prove the following Second Main Theorem for
 hypersurfaces being in N-subgeneral position.
\begin{theorem}\label{Th4}
Let $V\subset \C P^M$ be a smooth complex projective variety of
dimension $n\geq 1.$ Let f be an algebraically nondegenerate
holomorphic mapping of $\C$ into $V.$ Let  $D_1,\dots, D_q$
($V\not\subseteq D_j$) be hypersurfaces in $\C P^M$ of
degree $d_j,$ in $N$-subgeneral position in $V,$ where $N\geq n$ and $q\geq 2N-n+1.$ Then, for every $\epsilon>0,$ there exist  positive
integers $L_j\:(j=1,...,q)$ depending on $n, \deg V, N, d_j \:(j=1,...,q),q, \epsilon$ in an explicit way such that
\begin{align}
\Big\Vert(q-2N+n-1-\epsilon)T_f(r)\leq \sum_{j=1}^q\frac{1}{d_j}N_f^{[L_j]}(r,D_j).
\end{align}
\end{theorem}
The explicit bounds which we will get with the proof of Theorem \ref{Th4} are as follows:

\begin{proposition} \label{P5}
Assume without loss of generality that $\epsilon \leq 1$. Let $d$ be the least common
multiple of the  $d_j$'s. Put 
$$m_0=m_0(n,\deg V, N, d, q, \epsilon):= [4d^{n+1}q(2n+1)(2N-n+1)\deg V \cdot \frac{1}{\epsilon}]+1\,,$$
then 
\begin{align}
L_j \leq \big[ \frac{d_j \big( \big( 
\begin{array}{c}q+m_0-1\\m_0 \end{array}
\big) -1 \big)}{d} +1 \big]\:,
\end{align}
where we denote $[x]:=\max\{k\in
\Z: k\leq x\}$ for a real number $x.$

\end{proposition}

The proof of Theorem \ref{Th4} consists of three parts: In the first part (chapter 2), we extend the Nochka weights from the case of hyperplanes to the case of hypersurfaces. In the second part (chapter 4 until (\ref{3.9})) we reduce the case of hypersurfaces to the case of hyperplanes.
 The method in this part is based on the work of Evertse - Ferretti \cite{EF2}, Nochka \cite{N}, and Ru \cite{R3}.
  In the last part, we obtain an effective truncation for the counting functions. For this we devellop
  a new method using Hilbert weights, which is, in particular,  different from the method which is used in the case of nondegenerate holomorphic  curves in a complex projective space (see Dethloff-Tan \cite{DT}).  
  
  We also note that the proof of our Second Main Theorem remains valid  if more generally the hypersurfaces have Nochka weights.\\

Let us finally give an example for the special case $V=\C P^2.$ We consider three  quadrics $\Gamma_1,\Gamma_2,\Gamma_3$  in $\C P^2$ such that they have one common point $A_1.$ Let $A_2, A_3$ be distinct points in  $\C P^2\setminus \cup_{i=1}^3\Gamma_i.$  Let $B_i\in \Gamma_i\setminus(\Gamma_u\cup\Gamma_v)\;(\{i,u,v\}=\{1,2,3\})$ such that the lines $B_iA_2, B_iA_3$ are distinct and do not pass through  any intersection point of any pair of curves in $\cup_{1\leq s\leq i-1}\{B_sA_2, B_sA_3\}\cup\{\Gamma_1,\Gamma_2,\Gamma_3\}.$ Take three distinct lines $L_1,L_2,L_3$ which do not pass through any intersection point  of any pair of curves in $\cup_{1\leq i\leq 3}\{B_iA_2, B_iA_3\}\cup\{\Gamma_1,\Gamma_2,\Gamma_3 \}$ and $L_1,L_2,L_3$ have the common point $A_4$ which does not belong to  any $\Gamma_i,$ $B_iA_2,B_iA_3$ $(i=1,2,3).$ Set $\mathcal G_1:=\{\Gamma_1, \Gamma_2,\Gamma_3\},\;\mathcal G_i:=\{A_iB_1,A_iB_2,A_iB_3\}\; (i=2,3),$ and $\mathcal G_4:=\{L_1,L_2,L_3\}.$
Then the curves in the set $\mathcal G:=\cup_{i=1}^4\mathcal G_i$  are in 3-subgeneral position in $\C P^2.$ Hence, by Theorem \ref{Th4}, for any algebraically nondegenerate holomorphic curve $f$ in $\C P^2$ and for any $\epsilon>0,$
\begin{align*}
\Big\Vert(7-\epsilon)T_f(r)\leq \sum_{D\in\mathcal G}\frac{1}{\deg D}N_f(r,D).
\end{align*}
On the other hand, we  can not get the above inequality from the Second Main Theorem for hypersurfaces in general position (Theorem \ref{Th3}). In fact, for any $\mathcal G'\subset\mathcal G$ such that the curves in $\mathcal G'$ are in general position, it is clear that   $\#(\mathcal G'\cap\mathcal G_i)\leq 2$ for all $1\leq i\leq 4.$ So, $\# \mathcal G'\leq 8.$ We write  $\mathcal G=\cup_{i=1}^s\mathcal G_i,$ such that $\mathcal G_i\cap\mathcal G_j=\varnothing \;(1\leq i<j\leq s)$ and for any $i\in\{1,\dots,s\}$ the curves in $\mathcal G_i$ are in general position. We have  $\#\mathcal G_1+\cdots+\#\mathcal G_s=12$ and $\#\mathcal G_i\leq 8,\;(i=1,\dots,s).$ By Theorem \ref{Th3}, we get
\begin{align*}
\Big\Vert(\#\mathcal G_i-3-\epsilon)T_f(r)\leq \sum_{D\in\mathcal G_i}\frac{1}{\deg D}N_f(r,D), (i=1,\dots s).
\end{align*}
So by summing up over any partition of $\mathcal G=\cup_{i=1}^s\mathcal G_i,$ since such a partition must have at least two elements, we get at most a term $\Big\Vert(6-\epsilon)T_f(r)$
on the left hand side.

\section{Nochka weights for hypersurfaces in subgeneral position}
In this section, we shall prove the existence of the  Nochka weights for hypersurfaces in subgeneral position which was proved by Nochka for the case of hyperplanes. We mainly follow the ideas of Chen \cite{Ch}, Nochka \cite{N}, Ru-Wong \cite{RW}, and Vojta \cite{V2}. However, we have to pass some difficulties due to the fact that their methods are based on properties of linear algebra.
We finally would like to remark that the existence of Nochka weights for the case of infinitely many hyperplanes has been established by N. Toda \cite{To}.

 Let $V \subset \C P^M$ be a smooth projective variety of dimension $n.$ Throughout of this section, we consider $q$ hypersurfaces  $D_1,\dots,D_q \subset \C P^M$ in $N$-subgeneral position in $V,$ where $N\geq n$ and $q\geq N+1.$ Set $Q:=\{1,\dots,q\},$ $\text{codim} \varnothing:=n+1,\;c(\varnothing):=0,$ and $c(R):=\text{codim} \big(V \cap (\cap_{j\in R}D_j)\big),$ where the codimension is taken with respect to $V$ and $\varnothing\ne R\subseteq Q.$  It is easy to see that
\begin{remark}\label{R1}
$(i)$ For any $K\subseteq Q$, we have $c(K)\leq \# K,$ moreover,  $c(K)=\# K$ if and only if $\# K\leq n+1$ and the hypersurfaces $D_j\;(j\in K)$ are in general position in $V$.

$(ii)$ For $K\subseteq K'\subseteq Q,$ if $c(K')=\# K'$ then $c(K)=\#K.$
\end{remark}
\begin{lemma}\label{L21}
Let  $ K\subseteq R\subseteq Q$ such that $\# K=c(K).$ Then  there exists a set $K'$ such that  $K\subseteq K'\subseteq R$ and $c(K')=\# K'=c(R).$
\end{lemma}
\begin{proof}
 We have $\# K=c(K)\leq c(R).$ If $c(K)=c(R),$ then the lemma is trivial by taking $K'=K.$ If $c(K)<c(R),$ by induction,  it suffices to show that there exists $i\in R\setminus K$ such that $c(K\cup\{i\})=\#K+1\;(=c(K)+1).$

  Suppose that  $c(K\cup\{i\})\ne \#K+1=c(K)+1$ for every $i\in R\setminus K.$ Then  $c(K\cup\{i\})=c(K)$ for all $i\in R.$ If $K = \varnothing$ this is a contradiction,
  either, in the case $R \not= \varnothing$,  to the hypothesis of $N$-subgeneral position (including $V \not\subset D_i$), or, if $R=\varnothing$, to the hypothesis $c(K)<c(R)$.
  If $K \not= \varnothing$ this means that  $\text{dim}\big(V \cap D_i \cap (\cap_{j\in K}D_j)\big) =\text{dim}\big(V \cap (\cap_{j\in K}D_j)\big)$ for all $i\in R.$ Therefore, since $\{D_j,\;j\in Q\}$ are in $N$-subgeneral position, there exists an irreducible component $\sigma_K $ of $V \cap (\cap_{j\in K}D_j)$ with $\text{dim}\;\sigma_K=\text{dim}\big( V \cap (\cap_{j\in K}D_j)\big)$ such that $D_i$ contains $\sigma_K$ for all $i\in R\setminus K.$    Hence, we get $\text{dim}\big( V \cap (\cap_{j\in R}D_j)\big)=\text{dim}\big(V \cap (\cap_{j\in K}D_j)\big).$ This means that $c(R)=c(K).$  This is a contradiction. This completes the proof of Lemma \ref{L21}.
\end{proof}
\begin{lemma}\label{L22}
$(i)$ For  any subsets $R_1, R_2\subseteq Q,$ we have $$c(R_1\cup R_2)+c(R_1\cap R_2)\leq c(R_1)+c(R_2).$$

$(ii)$ For any $S_1\subseteq S_2\subseteq Q,$ we have $\# S_1-c(S_1)\leq \# S_2-c(S_2).$ Furthermore, if $\#S_2\leq N+1$ then $\# S_2-c(S_2)\leq N-n.$
\end{lemma}
\begin{proof}
Proof of $(i)$: By Lemma \ref{L21}, there exist subsets $K, K_1, K_3$ with  $K\subseteq R_1\cap R_2,$ $K\subseteq K_1\subseteq R_1,$ $K_1\subseteq K_3\subseteq R_1\cup R_2,$ such that
\begin{align*}
\# K=c(K)&=c(R_1\cap R_2),\; \# K_1=c(K_1)=c(R_1),\;\text{and}\\
&\# K_3= c(K_3)=c(R_1\cup R_2).
\end{align*}
   Set $K_2:=K_3\setminus K_1.$ Then $K_2\subseteq R_2.$ Indeed, otherwise there exists $i\in K_2\setminus R_2.$ Then $i\in R_1\setminus K_1.$ Therefore, $K_3\supseteq K_1\cup\{i\}\subseteq R_1.$  This implies that $c(R_1)\geq c(K_1\cup\{i\})=\#(K_1\cup\{i\})=c(K_1)+1=c(R_1)+1$ by Remark \ref{R1}. This is a contradiction. Hence, $K_2\subseteq R_2.$ Therefore, $K_2\cup K\subseteq R_2.$ On the other hand $K_2\cup K\subseteq K_3,$ and $K_2\cap K\subseteq K_2\cap K_1=(K_3\setminus K_1)\cap K_1=\varnothing.$ From these facts and by Remark \ref{R1} (ii) we get $c(R_2)\geq c(K_2\cup K)=\#(K_2\cup K)=\# K_2+\# K=(\#K_3-\# K_1)+\# K=c(R_1\cup R_2)-c(R_1)+c(R_1\cap R_2).$ Hence, the assertion (i) holds.

Proof of $(ii)$: By Lemma \ref{L21},  there exist  $S'_v\; (v=1,2)$ such that $S'_v\subseteq S_v,\; S'_1\subseteq S'_2$ and $\# S'_v=c(S'_v)=c(S_v).$ We have $(S'_2\setminus S'_1)\cap S_1=\varnothing.$ Indeed, otherwise there exists $i\not\in S'_1$ such that $S'_2\supseteq S'_1\cup\{i\}\subseteq S_1. $ Therefore, by Remark \ref{R1} (ii) we get $c(S_1)+1=\# S'_1+1= c(S'_1\cup\{i\})\leq c(S_1).$ This is a contradiction. Hence, $(S'_2\setminus S'_1)\cap S_1=\varnothing.$ Thus we have $S'_2\setminus S'_1\subseteq S_2\setminus S_1.$ Therefore, $c(S_2)-c(S_1)=\# S'_2-\#S'_1=\#(S'_2\setminus S'_1)\leq \#(S_2\setminus S_1)=\# S_2- \# S_1.$

If $\# S_2\leq N+1, $ then we choose $S_3$ such that $S_2\subseteq S_3\subseteq Q$ and $\#S_3=N+1.$ Since $D_j\; (j\in Q)$ are in $N$-subgeneral position, we have $c(S_3)=n+1.$ Therefore, $\#S_2-c(S_2)\leq \#S_3-c(S_3)=N-n.$
\end{proof}
For  $R_1\subsetneq R_2\subseteq Q,$ we set $\rho(R_1, R_2)=\frac{c(R_2)-c(R_1)}{\# R_2-\#R_1}.$ Then by Lemma \ref{L22}, we have $0\leq \rho(R_1,R_2)\leq 1.$
\begin{lemma}\label{L23}
Let $D_1,\dots,D_q$ be hypersurfaces in $N$-subgeneral position in $V,$  where $N\geq n$ and $q\geq 2N-n+1.$ Then, there exists a sequence of subsets $\varnothing:=R_0\subsetneq R_1\subsetneq\cdots\subsetneq R_s\subseteq Q:=\{1,\dots,q\}$ $(s\geq 0)$  satisfying the following conditions:

$(i)$ $c(R_s)<n+1,$

$(ii)$ $0<\rho(R_0, R_1)<\rho(R_1, R_2)<\cdots <\rho(R_{s-1}, R_s)<\frac{n+1-c(R_s)}{2N-n+1-\#R_s},$

$(iii)$ for any $R$ with $R_{i-1}\subsetneq R\subseteq Q\;(1\leq i\leq s),$ and $c(R_{i-1})<c(R)<n+1,$ we have that $\rho(R_{i-1}, R_i)\leq \rho(R_{i-1},R)$ and, moreover, if $\rho(R_{i-1}, R_i)=\rho(R_{i-1},R)$ then $\# R\leq\# R_i.$

$(iv)$ for any  $R$ with  $R_s\subsetneq R\subseteq Q,$  if  $c(R_s)<c(R)<n+1,$ then $\rho(R_s, R)\geq\frac{n+1-c(R_s)}{2N-n+1-\#R_s}.$
\end{lemma}
\begin{proof}
We start the proof by setting $R_0=\varnothing.$ It suffices to show that, under the assumption that there is a sequence $\varnothing:=R_0\subsetneq R_1\subsetneq\cdots\subsetneq R_s\subseteq Q$ satisfying  conditions (i), (ii) and (iii), it satisfies also the condition (iv) or, otherwise, there exists a subset $R_{s+1}$ such that the sequence $\varnothing:=R_0\subsetneq R_1\subsetneq\cdots\subsetneq R_{s+1}\subseteq Q:=\{1,\dots,q\}$ satisfies conditions (i), (ii) and (iii). In fact, if the latter case occurs, we can reach the desired conclusion after finitely many repetitions of these constructions.

 We now consider a sequence $\varnothing:=R_0\subsetneq R_1\subsetneq\cdots\subsetneq R_s\subseteq Q$  satisfying condition (i), (ii) and (iii). Assume that this sequence does not satisfy the condition (iv).
Set $\mathcal R:=\{R:\; R_s\subsetneq R\subseteq Q, c(R_s)<c(R)<n+1,\;\text{and}\; \rho(R_s, R)<\frac{n+1-c(R_s)}{2N-n+1-\#R_s}\}.$
 Then, we have $\mathcal R\ne\varnothing.$ Set $\rho:=\min\{\rho(R_s, R):\; R\in\mathcal R\}.$ We choose a set $R_{s+1}$ in $\mathcal R$  such that $\rho(R_{s}, R_{s+1})=\rho$ and $\# R_{s+1}$ is as big as possible.

We now prove that the sequence $R_0\subsetneq R_1\subsetneq\cdots\subsetneq R_{s+1}$ satisfies  conditions (i), (ii) and (iii).

$*$  It is clear that  $c(R_{s+1})<n+1,$ since $R_{s+1}\in \mathcal R.$

$*$ If $s\geq 1$, we have  $R_{s-1}\subsetneq R_{s+1}\subseteq Q,\; c(R_{s-1})\leq c(R_s)<c(R_{s+1})<n+1,$  and $\# R_{s+1}>\#R_{s}.$ Therefore, since the sequence  $R_0\subsetneq \cdots\subsetneq R_s$ satisfies the condition (iii), we have
 \begin{align}\label{no1}
 \rho(R_{s-1},R_s)<\rho(R_{s-1},R_{s+1}).
  \end{align}
On the other hand, for any $0\leq a\leq c,\; 0<b<d$ such that $\frac{a}{b}<\frac{c}{d},$ we have
\begin{align}\label{no2}
 \frac{a}{b}<\frac{c-a}{d-b}.
  \end{align}
Therefore, by (\ref{no1}) we have $\rho(R_{s-1},R_s)<\rho(R_{s},R_{s+1}).$
And if $s=0$, then we have $\rho (R_0,R_1) = \rho (\varnothing, R_1) =
\frac{c(R_1)}{\#R_1}>0$.

 \noindent Since $R_{s+1}\in \mathcal R,$ we get $\rho(R_s, R_{s+1})=\frac{c(R_{s+1})-c(R_s)}{\#R_{s+1}-\#R_s}<\frac{n+1-c(R_s)}{2N-n+1-\#R_s}.$ Hence, in both cases, by using the property (\ref{no2}), we get $\rho(R_s, R_{s+1})<\frac{n+1-c(R_{s+1})}{2N-n+1-\#R_{s+1}}$
 (observing that, by the hypothesis of $N$-subgeneral position in $V$, we get from $c(R_{s+1})<n+1$ that  $\#R_{s+1} \leq N< 2N-n+1$).

$*$ Let $R$ (if there exists any) such that $R_{s}\subsetneq R\subseteq Q$ and $c(R_{s})<c(R)<n+1.$ If $\rho(R_s, R)\geq\frac{n+1-c(R_s)}{2N-n+1-\#R_s},$ then $\rho(R_s, R_{s+1})=\rho<\rho(R_s, R).$ If $\rho(R_s, R)<\frac{n+1-c(R_s)}{2N-n+1-\#R_s},$ then $R\in\mathcal R.$  Therefore, by our choice of $R_{s+1}$ we have that $\rho(R_s, R_{s+1})\leq\rho(R_s, R),$ furthermore, if $\rho=\rho(R_s, R_{s+1})=\rho(R_s, R)$ then $\#R\leq \# R_{s+1}.$

From theses facts, we get that the sequence $R_0\subsetneq R_1\subsetneq\cdots\subsetneq R_{s+1}$ satisfies  conditions (i), (ii) and (iii). This completes the proof of Lemma \ref{L23}.
\end{proof}
\begin{proposition}\label{P24}
Let $D_1,\dots, D_q$ be hypersurfaces  in $N$-subgeneral position in $V$, where $N\geq n$ and  $q\geq 2N-n+1.$ Then, there exist constants $\omega(1),\dots,\omega(q)$ and $\Theta$ satisfying the following conditions:

$(i)$ $0<\omega(j)\leq\Theta\leq 1\;(1\leq j\leq q),$

$(ii)$ $\sum_{j=1}^q\omega(j)=\Theta (q-2N+n-1)+n+1,$

$(iii)$ $\frac{n+1}{2N-n+1}\leq\Theta\leq\frac{n+1}{N+1},$

$(iv)$ if $ R\subseteq Q$ and $0<\# R\leq N+1,$ then $\sum_{j\in R}\omega(j)\leq c(R).$
\end{proposition}
\begin{proof}
If $N=n,$ then $\omega(1)=\cdots=\omega(q)=1$ and $\Theta=1$ satisfy the conditions (i) to (iv). Assume that $N>n.$ Let $\{R_i\}_{i=0}^s$ be a sequence of subsets of $Q:=\{1,\dots,q\}$ satisfying the conditions (i) to (iv) of Lemma \ref{L23}.
By Lemma \ref{L23} (i) and by the ``N-subgeneral position" condition, we have 
\begin{align}\label{non}
\#R_s\leq N.
\end{align} 
Take a subset $R_{s+1}$ of $Q$ such that $\# R_{s+1}=2N-n+1\geq N+1$ and, hence,  $R_s\subsetneq R_{s+1}.$ Then we have $c(R_{s+1})=n+1.$

Set
\begin{equation*}
\Theta:=\rho(R_s,R_{s+1})=\frac{n+1-c(R_s)}{2N-n+1-\# R_s},\;\text{and}
\end{equation*}
\[\omega(j):=\begin{cases}\rho(R_{i},R_{i+1})&\mbox{if}\; j\in R_{i+1}\setminus R_{i} \;\mbox{for some\;} i\; \mbox{with}\; 0\leq i\leq s,\\
                       \Theta &\mbox{if}\; j\not\in R_{s+1}.\\
\end{cases}\]
By Lemma \ref{L23} (ii),  $\{\omega(j)\}_{j=1}^q$ and $\Theta$ satisfy the condition (i) of Proposition \ref{P24}.

We have
\begin{align*}
\sum_{j=1}^q\omega(j)&=\sum_{j\in Q\setminus R_{s+1}}\omega(j)+\sum_{i=0}^s\sum_{j\in R_{i+1}\setminus R_{i}}\omega(j)\\
&=\Theta(q-2N+n-1)+\sum_{i=0}^s\big(c(R_{i+1})-c(R_i)\big)\\
&=\Theta(q-2N+n-1)+c(R_{s+1})\\
&=\Theta(q-2N+n-1)+n+1.
\end{align*}
Thus, $\{\omega(j)\}_{j=1}^q$ and $\Theta$ satisfy the condition (ii) of Proposition \ref{P24}.

We next check the condition (iii). By (i) and (ii), we have
\begin{align*}
n+1=\sum_{j=1}^q\omega(j)-\Theta(q-2N+n-1)\leq\Theta(2N-n+1).
\end{align*}
By Lemma \ref{L22} (ii) we have
\begin{align*}
\Theta=\frac{n+1-c(R_{s})}{N+1+(N-n-\# R_s)}\leq\frac{n+1-c(R_s)}{N+1-c(R_s)}\leq\frac{n+1}{N+1}.
\end{align*}

Finally we check the condition (iv). Take an arbitrary subset $R$ of $Q$ with $0<\#R\leq N+1.$

{\bf Case 1: } $c(R\cup R_s)\leq n.$

Set
\[R'_i:=\begin{cases}R\cap R_i\;&\mbox{if}\; 0\leq i\leq s,\\
R\;&\mbox{if}\;i=s+1.\end{cases}\]
 We now prove that: for any $i\in\{1,\dots,s+1\},$ if $\# R'_{i}>\# R'_{i-1}$ then
 \begin{align}\label{no3}
 c(R'_i\cup R_{i-1})>c(R_{i-1})
 \end{align}
 and
\begin{align}\label{no4}
 \rho(R_{i-1}, R_{i})\leq \rho(R'_{i-1}, R'_{i}).
\end{align}

$*$ If $i=1$ then $c(R'_1\cup R_0)=c(R'_1)>0=c(R_0)$ (note that $R'_1\ne\varnothing,$ since $\# R'_1>\# R'_0=0).$

$*$ If $i\geq2,$ then since $\# R'_{i}>\# R'_{i-1},$ we have $\# (R'_{i}\cup R_{i-1})>\# R_{i-1}.$ On the other hand $c(R_{i-2})<c(R_{i-1})\leq c(R'_{i}\cup R_{i-1})\leq c(R\cup R_s)\leq n$ (note that $\rho(R_{i-2}, R_{i-1})>0).$ Therefore, by Lemma \ref{L23}, (iii) we have $\rho(R_{i-2}, R_{i-1})<\rho(R_{i-2}, R'_i\cup R_{i-1}).$
This means that
\begin{align*}
\frac{c(R_{i-1})-c(R_{i-2})}{\# R_{i-1}-\# R_{i-2}}<\frac{c(R'_i\cup R_{i-1})-c(R_{i-2})}{\# (R'_i\cup R_{i-1})-\# R_{i-2}}.
\end{align*}
Therefore, since $\# R_{i-1}<\# (R'_{i}\cup R_{i-1}),$ we have $c(R_{i-1})<c(R'_i\cup R_{i-1}).$ We get (\ref{no3}).

We next prove (\ref{no4}). By (\ref{no3}), we have $c(R_{i-1})<c(R'_i\cup R_{i-1})\leq c(R\cup R_s)\leq n.$ Hence, by Lemma \ref{L23}, (iii) for the case $1\leq i\leq s$ and (iv) for the case $i=s+1,$ we have
\begin{align*}
\rho(R_{i-1}, R_i)\leq \rho(R_{i-1},R'_i\cup R_{i-1})\;\text{for all}\; i\in\{1,\dots,s+1\},
\end{align*}
(note that $\rho(R_s,R_{s+1})=\frac{n+1-c(R_s)}{2N-n+1-\# R_s}).$

\noindent Therefore, by Lemma \ref{L22}, (i) we have
\begin{align*}
\rho(R_{i-1}, R_i)&\leq \rho(R_{i-1},R'_i\cup R_{i-1})\\
&=\frac{c(R'_i\cup R_{i-1})-c(R_{i-1})}{\#(R'_i\cup R_{i-1})-\#R_{i-1}}\leq\frac{c(R'_i)-c(R'_i\cap R_{i-1})}{\#(R'_i\cup R_{i-1})-\# R_{i-1}}\\
&=\frac{c(R'_i)-c(R'_{i-1})}{\# R'_i-\# (R'_{i}\cap R_{i-1})}=\frac{c(R'_i)-c(R'_{i-1})}{\# R'_i-\# R'_{i-1}}\\
&=\rho(R'_{i-1}, R'_i),
\end{align*}
(note that $R'_{i-1}=R'_i\cap R_{i-1}$). We get (\ref{no4}).

By (\ref{no4}), we   get that
\begin{align}\label{no5}
\omega(j)\leq \rho(R'_{i-1}, R'_i)\;\text{for all}\; j\in R'_i\setminus R'_{i-1}\; (1\leq i\leq s+1).
\end{align}
In fact, for $j\in R'_{s+1}\setminus R'_s$ we have
$\omega(j)\leq\Theta=\rho(R_s,R_{s+1})\leq \rho(R'_s,R'_{s+1}),$ and for $j\in R'_i\setminus R'_{i-1}\subseteq R_i\setminus R_{i-1}\; (1\leq i\leq s)$ we have $\omega(j)=\rho(R_{i-1}, R_i)\leq \rho(R'_{i-1},R'_i).$

 By (\ref{no5}), we have
\begin{align*}
\sum_{j\in R}\omega(j)&=\sum_{i=1}^{s+1}\sum_{j\in R'_i\setminus R'_{i-1}}\omega(j)\\
&\leq\sum_{i=1}^{s+1}(\# R'_i-\# R'_{i-1})\cdot\rho(R'_{i-1},R'_i)\\
&=c(R'_{s+1})-c(R'_0)=c(R).
\end{align*}
Therefore, the assertion  (iv) holds in this case.

{\bf Case 2:} $c(R\cup R_s)=n+1.$
\noindent By Lemma \ref{L22}, and since $\# R\leq N+1,$ we have
\begin{align*}
&\# R\leq c(R)+N-n,\;\text{and}\\
 &n+1-c(R_s)=c(R\cup R_s)-c(R_s)\leq c(R)-c(R\cap R_s)\leq c(R).
\end{align*}
Therefore, by the assertion (i), by the definition of $\Theta$ and by Lemma \ref{L22} (ii), 
applied to $R_s$ and by using $(\ref{non})$, we have
\begin{align*}
\sum_{j\in R}\omega(j)\leq\Theta\#R&\leq\Theta(c(R)+N-n)\\
&=\Theta c(R)\big(1+\frac{N-n}{c(R)}\big)\\
&\leq \Theta c(R)\big(1+\frac{N-n}{n+1-c(R_s)}\big)\\
&=c(R)\frac{N+1-c(R_s)}{2N-n+1-\#R_s}\\
&\leq c(R).
\end{align*}
This completes the proof of Proposition \ref{P24}.
\end{proof}
\begin{definition}
We call constants $\omega(j)\;(1\leq j\leq q)$ respectively $\Theta$ with the properties (i) to (iv) in Proposition \ref{P24}   Nochka weights respectively  Nochka constant for hypersurfaces $D_1,\dots,D_q$  in $N$-subgeneral position in $V$, where $N\geq n$ and  $q\geq 2N-n+1.$ 
\end{definition}
\begin{theorem}\label{Th25}
Let $D_1,\dots, D_q$ be hypersurfaces in $N$-subgeneral position in $V$ and $\omega(1),\dots,\omega(q)$ be Nochka weights for them, where $N\geq n$ and $q\geq 2N-n+1.$ Consider an arbitrary subset R of $Q:=\{1,\dots,q\}$ with $0<\# R\leq N+1$ and $c^*:=c(R),$ and arbitrary nonnegative real constants $E_1,\dots,E_q.$ Then, there exist $j_1,\dots,j_{c^*}\in R$ such that  the hypersurfaces $D_{j_1},\dots,D_{j_{c^*}}$ are in general position and
\begin{align*}
\sum_{j\in R}\omega(j)E_j\leq\sum_{i=1}^{c^*}E_{j_i}.
\end{align*}
\end{theorem}
\begin{proof}
Without loss of the generality, we may assume that $E_1\geq E_2\geq\cdots\geq E_q.$
We shall choose indices $j_i' s$ in $R$ by induction on $i.$ We first choose 
\begin{align*}
 j_1:=\min\{t\in R\}
\end{align*}
 and set  $ K_1:=\{k\in R:\;c(\{j_1, k\})=c(\{j_1\})=1\}.$  Next, choose
 \begin{align*}
 j_2:=\min\{t\in R\setminus K_1\}
 \end{align*}
 and set $ K_2:=\{k\in R:\; c(\{j_1, j_2, k\})= c(\{j_1, j_2\})=2\}.$ Similarly, choose
\begin{align*}
 j_3:=\min\{t\in R\setminus K_2\}
\end{align*}
and set $ K_3:=\{k\in R:\; c(\{j_1, j_2, j_3,k\})= c(\{j_1, j_2, j_3\})=3\}.$
By Lemma \ref{L21}, we can repeat this process until $j_{c^*}$ and $K_{c^*}.$ Then, we have $K_1\subsetneq K_2\subsetneq\cdots\subsetneq K_{c^*}=R.$   We have $\text{dim}(D_{j_1}\cap\cdots\cap D_{j_i}\cap D_k)=\text{dim}(D_{j_1}\cap\cdots\cap D_{j_i}),$ for all  $k\in K_i.$ Therefore, by the ``$N$-subgeneral position" condition, for any $i\in\{1,\dots,c^*\},$ there exists an irreducible components $\sigma_i$ of $D_{j_1}\cap\cdots\cap D_{j_i}$ with $\text{dim}\sigma_i=\text{dim}(D_{j_1}\cap\cdots\cap D_{j_i})$ such that we have that $D_k$ contains $\sigma_i$ for all $k\in K_i.$  Thus, $\text{dim}\cap_{j\in K_i}D_j=\text{dim}(D_{j_1}\cap\cdots\cap D_{j_i})=n-i.$ Then $c(K_i)=i$ for all $i\in\{1,\dots,c^*\}.$

Set $K_0:=\varnothing$ and $a_i:=\sum_{j\in K_{i}\setminus K_{i-1}}\omega(j),\;i=1,\dots,c^*.$  Therefore, by Proposition \ref{P24}, we get
\begin{align*}
\sum_{k=1}^ia_i=\sum_{j\in K_i}\omega(j)\leq c(K_i)= i\;\; \text{for all}\; i\in\{1,\dots,c^*\}.
\end{align*}
 On the other hand, for any $1\leq i\leq c^*$ we have  $E_j\leq E_{j_i}$ for all $j\in K_{i}\setminus K_{i-1}(\subseteq R\setminus K_{i-1}).$  Thus, we have
\begin{align*}
\sum_{j\in R}\omega(j)E_j&=\sum_{i=1}^{c^*}\sum_{j\in K_{i}\setminus K_{i-1}}\omega(j)E_j\\
&\leq \sum_{i=1}^{c^*}\sum_{j\in K_{i}\setminus K_{i-1}}\omega(j)E_{j_i}=\sum_{i=1}^{c^*}a_iE_{j_i}\\
&=\sum_{i=1}^{c^*-1}(a_1+\cdots +a_i)(E_{j_i}-E_{j_{i+1}})+(a_1+\cdots+a_{j_{c^*}})E_{j_{c^*}}\\
&\leq \sum_{i=1}^{c^*-1} i(E_{j_i}-E_{j_{i+1}})+c^*E_{j_{c^*}}\\
&=\sum_{i=1}^{c^*}E_{j_i}.
\end{align*}
This completes the proof of Theorem \ref{Th25}.
\end{proof}

\section{Some lemmas}
Let $X\subset \C P^M$ be a projective variety of dimension $n$ and
degree $\bigtriangleup.$ Let $I_X$ be the prime ideal in $\C
[x_0,\dots,x_M]$ defining $X.$ Denote by $\C [x_0,\dots,x_M]_m$ the
vector space of homogeneous polynomials in $\C [x_0,\dots,x_M]$ of
degree $m$ (including $0$). Put $I_X(m):=\C [x_0,\dots,x_M]_m\cap
I_X.$

The Hilbert function $H_X$ of $X$ is defined by
\begin{align}\label{h1}
H_X(m):=\dim \C [x_0,\dots,x_M]_m\diagup I_X(m).
\end{align}
In particular we have $H_X(m) \leq \big(
\begin{array}{c}M+m\\M \end{array} \big)$.
By the usual theory of Hilbert polynomials, we have
\begin{align}\label{h2}
H_X(m):=\bigtriangleup\cdot\frac{m^n}{n!}+O(m^{n-1}).
\end{align}
We also need the following result, which should be well known, but since we do not know
a good reference, we add a short proof:
\begin{lemma}\label{m+1}
For $n \geq 1$, we have $H_X(m) \geq m+1$ for all $m \geq 1$.
\end{lemma}
\begin{proof}
Using the notations introduced above, we first observe that there exists some $x_i$
which is not identically zero on $X$, without loss of generality
we may assume that it is $x_0$. It suffices to prove the following

CLAIM: For all $m \geq 1$ there exists $i \in \{1,...,M\}$ such that for
all $c_{ij} \in \C$ which are not all zero we haveÊ
$$\sum_{j=0}^m c_{ij}x_0^{m-j}x_i^j \not\equiv 0 \; \text{on} \; X.$$

In fact, if the claim is true, it means that no (nontrivial) complex linear combination
of the $m+1$ monomials $x_0^{m-j}x_i^j, \: j=0,...,m$ vanishes identically
on $X$, and, hence, can be contained in $I_X(m)$. So $H_X(m) \geq m+1$.

Assume that the claim does not hold. Then there exists
$m \geq 1$ such that for all $i \in \{1,...,M\}$ there exist  $c_{ij} \in \C$ which are not all zero so that we haveÊ
$$\sum_{j=0}^m c_{ij}x_0^{m-j}x_i^j \equiv 0\;  \text{on}\;  X.$$
Dividing by $x_0^m$ we get thatÊ
$$\sum_{j=0}^m c_{ij}(\frac{x_i}{x_0})^j \equiv 0\;  \text{on} \; X.$$
This means that the rational functions $\frac{x_i}{x_0}, \:i=1,...,M$ on $X$
are all algebraic over $\C$. Since the subset of rational functions on $X$ which are algebraic over $\C$ forms a subfield of the function field $\C (X)$ of $X$ and since (by what we saw above) this subfield contains the rational functions
$\frac{x_i}{x_0}, \:i=1,...,M$ on $X$, which  generate $\C (X)$ as
a field, this means that $\C (X)$ Êover $\C$ is an algebraic field extension.
So the transcendence degree of $\C (X)$ over $\C$ is zero.
But by a well know theorem (Hartshorne \cite{Ha} p.17), observing that we have
$\C (X) = \C(X_0)$ and $\dim X = \dim X_0$ if $X_0= X \cap \{x_0\not= 0\}$
is one affine chart of $X$, we getÊ
$$0 = \text{transcendence degree}(\C(X)) = \dim X.$$
With other words, if $n= \dim X \geq 1$, we get a contradiction, proving the claim.
\end{proof}

For each tuple $c=(c_0,\dots,c_M)\in\R_{\geq 0}^{M+1},$ and
$m\in\N,$ we define the $m$-th Hilbert weight $S_X(m,c)$ of $X$ with
respect to $c$ by
\begin{align*}
S_X(m,c):=\max\sum_{i=1}^{H_X(m)}I_i\cdot c,
\end{align*}
where $I_i=(I_{i0},\dots, I_{iM})\in\N_0^{M+1}$ and the maximum is
taken over all sets $ \{x^{I_i}=x_0^{I_{i0}}\cdots x_M^{I_{iM}}\}$
whose residue classes modulo $I_X(m)$ form a basis of the vector
space $\C [x_0,\dots,x_M]_m\diagup I_X(m).$
\begin{lemma}\label{L2.1}
Let $X\subset \C P^M$ be an algebraic variety of dimension $n$ and
degree $\bigtriangleup.$ Let $m>\bigtriangleup$ be an integer and
let $c=(c_0,\dots,c_M)\in\R_{\geq 0}^{M+1}.$ Let $\{i_0,\dots,
i_{n}\}$ be a subset of $\{0,\dots, M\}$ such that
$\{x=(x_0:\cdots:x_M)\in\C P^M: x_{i_0}=\cdots=x_{i_n}=0\}\cap
X=\varnothing.$ Then
\begin{align*}
\frac{1}{mH_X(m)}S_X(m,c)\geq\frac{1}{(n+1)}(c_{i_0}+\cdots
+c_{i_n})-\frac{(2n+1)\bigtriangleup}{m}\cdot\max_{0\leq i\leq
M}c_i.
\end{align*}
\end{lemma}
\begin{proof} We refer to \cite{EF}, Theorem 4.1, and \cite{EF2}, Lemma 5.1 (or \cite{R3}, Theorem 2.1 and Lemma 3.2).
\end{proof}
\begin{lemma}[Theorem 2.3 of  \cite{R1}]\label{L31}
Let f be a linearly nondegenerate holomorphic mapping of $\C$ into
$\C P^n$ and let $\{H_j\}_{j=1}^q$ be arbitrary hyperplanes in $\C
P^n.$  Then for every $\epsilon,$
\begin{align*}
 \Big\Vert\int_0^{2\pi}\max_{K\in\mathcal K}\sum_{j\in
K}\log \frac{\Vert f(re^{i\theta})\Vert\cdot\Vert H_j\Vert}{|H_j(f(re^{i\theta}))|}\frac{d\theta}{2\pi}+N_{W(f)}(r)
\leq (n+1+\epsilon)T_f(r).
\end{align*}
where  $\mathcal K$ is the set of
all subsets $K\subset\{1,\dots,q\}$ such that $\#K=n+1$ and the hyperplanes
$H_j,\;j\in K$ are in general position, $W(f)$ is the Wronskian
of $f,$ and $\Vert H_j\Vert$ is the maximum of absolute values of the coefficients of $H_j.$
\end{lemma}
\begin{lemma}[Propositions 4.5 and 4.10 of \cite{b5}]\label{L2.2}
Let $f$ be a linearly nondegenerate holomorphic mapping of $\C$ into
$\C P^M$ with  reduced representation $f=(f_0:\cdots :f_M).$ Let
$W(f)=W(f_0,\dots,f_M)$ be the Wronskian of $f.$ Then
\begin{align*}
\nu_{\frac{f_0\cdots f_M}{W(f)}}\leq \sum_{i=0}^M \min\{\nu_{f_i},
M\}.
\end{align*}
\end{lemma}

\section{Proof of Theorem \ref{Th4}}
We first prove Theorem \ref{Th4} for the case where all the $Q_j\:(j=1,...,q)$ have the same degree
$d$.

Since $D_1,\dots,D_q$ are in $N$-subgeneral position in $V,$  we have $\cap_{j=1}^qD_j\cap V=\varnothing.$
We define a map $\Phi:V\longrightarrow \C P^{q-1}$ by
 $\Phi(x)=(Q_1(x):\cdots:Q_q(x)).$ Then $\Phi$ is
 a finite morphism (see \cite{Sha}, Theorem 8, page 65). We have that $Y:=im\Phi$ is a complex projective subvariety of $\C P^{q-1}$
 and  $\dim Y=n$ and 
  \begin{align}\label{deg}
 \bigtriangleup:=\deg Y\leq d^n\cdot \deg
 V.
  \end{align}
 This follows, in the same way as  \cite{Sha}, Theorem 8, page 65, from the
 fact that $\Phi:V\longrightarrow \C P^{q-1}$ is the composition of the restriction of  the 
 $d$-uple embedding $\rho_d|_V : V\longrightarrow \C P^{L-1}$ to $V$ (with
 $L= \big( 
 \begin{array}{c} M+d\\M \end{array} 
 \big)$) with the linear projection $p: \C P^{L-1} \longrightarrow \C P^{q-1}$, defined by the linear forms  $Q_1,...,Q_q$ in the monomials of degree $d$,
 since we have:
 $$\deg Y = \deg \Phi (V) \leq \deg \rho_d|_V (V) \leq d^n\cdot \deg V.$$
  It is clear that for any $1\leq i_0<\cdots <i_n\leq q$ such that $\cap_{i=0}^n D_{j_i}\cap V=\varnothing,$ we have
 \begin{align}\label{new1}
 \{y=(y_1:\cdots :y_q)\in \C P^{q-1}:\;
 y_{i_0}=\cdots=y_{i_n}=0\}\cap Y=\varnothing.
 \end{align}

 For a positive integer $m,$ denote by $\{I_1,\dots,I_{q_m}\}$ the set  of all $I_i:=(I_{i1},\dots,I_{iq})\in\N_0^q$ with
 $I_{i1}+\cdots+I_{iq}=m.$ We have $q_m:=\binom{q+m-1}{m}.$

 Let $F$ be a holomorphic mapping of $\C$ into $\C P^{q_m-1}$ with the reduced representation
  $F=\big(Q_1^{I_{11}}(f)\cdots Q_q^{I_{1q}}(f):\cdots:Q_1^{I_{q_m1}}(f)\cdots
  Q_q^{I_{q_mq}}(f)\big),$ (note that $Q_1^m(f),\dots, Q_q^m(f)$
  have no common zero point).

 Define an isomorphism between vector spaces, $\Psi:
 \C[z_1,\dots,z_{q_m}]_1\longrightarrow \C[y_1,\dots,y_{q}]_m$ by
 $\Psi(z_i):=y^{I_i}\;\;(i=1,\dots,q_m).$
Consider the vector space $\mathcal H
:=\{H\in\C[z_1,\dots,z_{q_m}]_1: H(F)\equiv 0$ \}.  Then $F$ is a
linearly nondegenerate mapping of $\C$ into the complex projective
space $P:= \cap_{H\in\mathcal H}
  \{H=0\}\subset \C P^{q_m-1},$ and we will from now on, by abuse of notation, consider $F$ to be this linearly nondegenerate map $F:\C \rightarrow P$.

For any linear form $H \in \C[z_1,\dots,z_{q_m}]_1,$ since $f$ is
algebraically nondegenerate, we have that $H\in\mathcal H$ if only
if
  $$H(Q_1^{I_{11}}(x)\cdots Q_q^{I_{1q}}(x),\cdots, Q_1^{I_{q_m1}}(x)\cdots Q_q^{I_{q_mq}}(x))\equiv 0
  \;\textit{on\;}
  V.$$
  This is possible if and only if  $\Psi(H)(y):=H(y^{I_1},\cdots,y^{I_{q_m}})\equiv 0$ on
  $Y.$ Therefore, we get that $\Psi(\mathcal H)=(I_Y)_m.$ On the
  other hand
  $\Psi$ is an isomorphism. Hence, we have
  \begin{align}\label{new2}
  \dim P=\dim\bigcap\limits_{H\in\mathcal H}
  \{H=0\}&=q_m-1-\dim \mathcal H\notag\\
  &=q_m-1-\dim (I_Y)_m=H_Y(m)-1.
\end{align}
   We define hyperplanes $H_j\;(j=1,\dots,q_m)$ in the complex projective space
$P$ by $H_j:=\{(z_1:\dots :z_{q_m})\in \C P^{q_m-1}:\;z_j=0\}\cap P,$ (these intersections are not empty by B\'ezout's theorem, and they are proper algebraic subsets of $P$ since $V\not\subset D_k,\; 1\leq k\leq
  q).$

  Denote by $\mathcal L$ the set of all  subsets $J$ of
  $\{1,\dots,q_m\}$ such that $\#J=H_Y(m)$ and the hyperplanes $H_j, j\in J,$ are in general position in $P.$
    Since $\Psi$ is an
  isomorphism and $\Psi(\mathcal H)=I_Y(m),$ $\mathcal L$ is also the set of all subsets $J$ of $\{1,\dots,q_m\}$
  such that
  $\{y^{I_j},\; j\in J\}$ is a basis of $\C[y_1,\dots,y_{q}]_m\diagup I_Y(m).$

  For each $j\in\{1,\dots,q\}$ and $k\in\{1,\dots,q_m\},$ we put
  \begin{align*}
 E_{D_j}(f)=\log\frac{\Vert f\Vert^d\cdot \Vert Q_j\Vert}{|Q_j(f)|}\geq 0\;\;\text{and}\;\; E_{H_k}(F)=\log\frac{\Vert F\Vert\cdot \Vert H_k\Vert}{|H_k(F)|}\geq 0,
  \end{align*}
 where $\Vert Q_j\Vert$ (respectively $\Vert H_k\Vert$) is the maximum of absolute values of the coefficients of $Q_j$ (respectively $H_k$). They are continuous functions with values
 in $\R_{\geq 0} \cup \{+ \infty\}$ which take the value $+\infty$ only on discrete subsets of $\C$.

  Denote by $\mathcal K$ the set of all subsets $K$ of $\{1,\dots,q\}$ such that $\#K=n+1$ and $\cap_{j\in K}D_j\cap V=\varnothing.$ Let $\mathcal N$ be the set of all subsets $J\subset\{1,\dots,q\}$ with $\#J=N+1.$
   Let $\{\omega(j)\}_{j=1}^q$ and $\Theta$ be Nochka weights and Nochka constant for the hypersurfaces $D_j$ in $N$-subgeneral position in $V.$ By Theorem \ref{Th25}, for any $z\in \C$ and any  $J\in\mathcal N,$  there exists a subset $K(J,z)\in\mathcal K,$ such that
  \begin{align}\label{a1}
  \sum_{j\in J}\omega(j)E_{D_{j}}(f(z))\leq \sum_{j\in K(J,z)}E_{D_j}(f(z)).
  \end{align}

  For any $J\in\mathcal N,$ since the hypersurfaces $D_j\;(j=1,\dots,q)$ are in $N$-subgeneral position in $V,$  the function $\lambda_J(x):=\frac{\max_{j\in J}|Q(x)|}{\Vert x\Vert^d}$ is continuous on $V$ and $\lambda_J(x)>0$ for all $x\in V.$ On the other hand, $V$ is compact, so there exist positive constants $c_J, c'_J$ such that $c'_J\geq\lambda_J(f(z))\geq c_J$ for all $z\in\C.$ This implies that
  \begin{align}\label{a5}
  d\cdot\log\Vert f\Vert=\max_{j\in J}\log|Q(f)|+O(1),\;\text{for all}\;J\in\mathcal N.
  \end{align}
  Therefore, there exists a positive constant $c$ such that
  \begin{align*}
  \min_{\{j_1,\dots, j_{q-N-1}\}}\sum_{i=1}^{q-N-1} E_{D_{j_i}}(f)\leq c.
  \end{align*}
  Then, we have
  \begin{align}\label{a2}
  \sum_{j=1}^q\omega(j) E_{D_j}(f)\leq\max_{J\in\mathcal N} \sum_{j\in J}\omega(j) E_{D_j}(f)+O(1).
  \end{align}
  By (\ref{a1}) and (\ref{a2}), for every $z\in\C,$ we have
  \begin{align*}
  \sum_{j=1}^q\omega(j) E_{D_j}(f(z))&\leq \max_{J\in\mathcal N}\sum_{j\in K(J,z)}E_{D_j}(f(z))+O(1)\notag\\
  &\leq \max_{K\in\mathcal K}\sum_{j\in K}E_{D_j}(f(z))+O(1).
  \end{align*}
  This implies that
  \begin{align}\label{a9}
  \sum_{j=1}^q\omega(j)d\log\Vert f\Vert-\sum_{j=1}^q\omega(j)\log|Q_j(f)|&\leq\sum_{j=1}^q\omega(j)E_{D_j}(f)+O(1)\notag\\
  &\leq\max_{K\in\mathcal K}\sum_{j\in K}E_{D_j}(f)+O(1).
  \end{align}
  Applying integration on the both sides of (\ref{a9}), using Proposition \ref{P24} and Jensen's formula, we get
  \begin{align}\label{a6}
  d\big(\Theta(q-2N+n-1)&+n+1\big)T_f(r)-\sum_{j=1}^q\omega(j)N_f(r,D_j)\notag\\
  &\leq \int_0^{2\pi}\max_{K\in\mathcal K}\sum_{j\in
K}E_{D_j}(f(re^{i\theta}))\frac{d\theta}{2\pi}+O(1).
  \end{align}

  Since $Im F\subset P$ and  $\{Q_1^{I_{i1}}(f)\cdots
  Q_q^{I_{iq}}(f),\;1\leq i\leq q_m\}$ have no common zero point, for every $J\in \mathcal L,$
  the holomorphic functions  $\{Q_1^{I_{i1}}(f)\cdots
  Q_q^{I_{iq}}(f),\;i\in J\}$ also have no common zero point.

  \noindent Then, for every $J\in \mathcal L,$ we have
  \begin{align*}
  \Vert F\Vert=\max_{i\in J} |H_i(F)|+O(1)&=\max_{i\in J}| Q_1^{I_{i1}}(f)\cdots
  Q_q^{I_{iq}}(f)| +O(1)\\
  &\leq \Vert
  f\Vert^{dm}+O(1).
  \end{align*}
  This implies that
  \begin{align}\label{3.2}
  T_F(r)\leq dm\cdot T_f(r)+O(1).
  \end{align}

 For every $J\in\mathcal L$ and $i\in J,$ we have
  \begin{align}\label{3.3}
  E_{H_i}(F)&=\log\frac{\Vert F\Vert\cdot\Vert H_i\Vert}{\mid
  H_i(F)\mid}=\log\frac{\Vert F\Vert}{|
  Q_1^{I_{i1}}(f)\cdots Q_q^{I{iq}}(f)|}+O(1)\notag\\
  &=\log\frac{\Vert f\Vert^{dm}}{|
  Q_1^{I_{i1}}(f)\cdots Q_q^{I{iq}}(f)|}-dm\log\Vert
  f\Vert+\log\Vert F\Vert +O(1)\notag\\
  &=\sum_{1\leq j\leq q}I_{ij}E_{D_j}(f)-dm\log\Vert
  f\Vert+\log\Vert F\Vert +O(1).
  \end{align}

  Let
  $c_z:=(E_{D_1}(f(z)),\cdots,E_{D_q}(f(z)))$
   for every $z\in\C \setminus D,$ where $D$ denotes  the discrete subset where one of these functions takes the value $+\infty$.
   By the definition
  of the Hilbert weight, there exists a subset $J_z\in\mathcal L$ such that
  \begin{align}\label{3.4}
  S_Y(m,c_z)=\sum_{i\in J_z}I_i\cdot c_z.
  \end{align}
  By (\ref{new1}) and by Lemma \ref{L2.1}, for every $m>\bigtriangleup$ and $K\in\mathcal
K,$ we have
\begin{align}\label{3.5}
  \frac{S_Y(m,c_z)}{mH_Y(m)}\geq\frac{1}{n+1}\sum_{j\in
  K}E_{D_j}(f(z))-\frac{(2n+1)\bigtriangleup}{m}\max_{1\leq
  j\leq q}E_{D_j}(f(z)).
  \end{align}
Then, by (\ref{3.3}), (\ref{3.4}) and (\ref{3.5}), for every
$K\in\mathcal K,\; z\in \C \setminus D,$ we have
\begin{align}\label{3.6}
\frac{1}{(n+1)}\sum_{j\in
  K}&E_{D_j}(f(z))\leq \frac{S_Y(m,c_z)}{mH_Y(m)}+\frac{(2n+1)\bigtriangleup}{m}\max_{1\leq
  j\leq q}E_{D_j}(f(z))\notag\\
  &=\frac{\sum_{i\in J_z}I_i\cdot c_z}{mH_Y(m)}+\frac{(2n+1)\bigtriangleup}{m}\max_{1\leq
  j\leq q}E_{D_j}(f(z))\notag\\
  &=\frac{1}{mH_Y(m)}\sum\limits_{\begin{matrix} \scriptstyle{i\in J_z}\cr
\noalign{\vskip-0.15cm} \scriptstyle{1 \leq j \leq q}\end{matrix}}
I_{ij}E_{D_j}(z)+\frac{(2n+1)\bigtriangleup}{m}\max_{1\leq
  j\leq q}E_{D_j}(f(z))\notag\\
  &=\frac{1}{mH_Y(m)}\sum_{i\in J_z}E_{H_i}(F(z))+d\log\Vert
  f(z)\Vert\notag-\frac{1}{m}\log\Vert F(z)\Vert\notag\\
  &\quad\quad\quad+\frac{(2n+1)\bigtriangleup}{m}\max_{1\leq
  j\leq q}E_{D_j}(f(z))+O(1)\notag\\
  &\leq \frac{1}{mH_Y(m)}\max_{L\in\mathcal L}\sum_{i\in L}E_{H_i}(F(z))+d\log\Vert
  f(z)\Vert-\frac{1}{m}\log\Vert F(z)\Vert\notag\\
  &\quad\quad\quad+\frac{(2n+1)\bigtriangleup}{m}\sum_{1\leq
  j\leq q}E_{D_j}(f(z))+O(1).
\end{align}

 \noindent This implies that, for every $z\in\C \setminus D,$
\begin{align*}
\max_{K\in\mathcal K}\frac{1}{(n+1)}\sum_{j\in
K}E_{D_j}(f(z))&\leq \frac{1}{mH_Y(m)}\max_{L\in\mathcal
L}\sum_{i\in L}E_{H_i}(F(z))+d\log\Vert
  f(z)\Vert\notag\\
  &-\frac{1}{m}\log\Vert F(z)\Vert
  +\frac{(2n+1)\bigtriangleup}{m}\sum_{1\leq
  j\leq q}E_{D_j}(z)+O(1),
\end{align*}
and by continuity this then holds for all $z \in \C$.
So, by integrating and by (\ref{a6}), we get
\begin{align}\label{a10}
d\big(\Theta(q&-2N+n-1)+n+1\big)T_f(r)-\sum_{j=1}^q\omega(j)N_f(r,D_j)\notag\\
&\leq \frac{n+1}{mH_Y(m)}\int_0^{2\pi}\max_{L\in\mathcal
L}\sum_{i\in L}E_{H_i}(F(re^{i\theta}))\frac{d\theta}{2\pi}+d(n+1)T_f(r)-\frac{n+1}{m}T_ F(r)\notag\\
  &\quad\quad\quad+\frac{(2n+1)(n+1)\bigtriangleup}{m}\sum_{1\leq
  j\leq q}\int_0^{2\pi}E_{D_j}(re^{i\theta})\frac{d\theta}{2\pi}+O(1).
\end{align}
By (\ref{new2}) and Lemma \ref{L31} (with $\epsilon =1),$  we have
\begin{align}\label{3.8}
\Big\Vert \frac{n+1}{mH_Y(m)}\int_0^{2\pi}\max_{L\in\mathcal
L}&\sum_{i\in
L}E_{H_i}(F(re^{i\theta}))\frac{d\theta}{2\pi}\notag\\
  &\leq\frac{(n+1)(H_Y(m)+1)}{mH_Y(m)}T_F(r)-\frac{n+1}{mH_Y(m)}N_{W(F)}(r).
\end{align}
For each $j\in\{1,\dots,q\},$ by Jensen's formula,  we have
\begin{align}\label{a11}
\int_0^{2\pi}E_{D_j}(re^{i\theta})\frac{d\theta}{2\pi}&\leq d\int_0^{2\pi}\log\Vert f(re^{i\theta})\Vert\frac{d\theta}{2\pi}-\int_0^{2\pi}\log|Q_j(re^{i\theta})|\frac{d\theta}{2\pi}+O(1)\notag\\
&\leq dT_f(r)-N_f(r,D_j)+O(1)\leq dT_f(r)+O(1).
\end{align}

For an arbitrary $\epsilon>0,$ we choose
$$m:= [4d^{n+1}q(2n+1)(2N-n+1)\deg V \cdot \frac{1}{\epsilon}]+1\,.$$
Then, assuming without loss of generality that $\epsilon \leq 1$, by  (\ref{deg}), by Lemma \ref{m+1} and  by Proposition \ref{P24} (iii) we have $m >  \bigtriangleup$,
which we assumed for (\ref{3.5}), and
\begin{align}\label{cond}
\frac{(2n+1)(n+1)dq\bigtriangleup}{m}<\frac{\Theta\epsilon}{4} \; {\rm and}\;
\frac{(n+1)d}{H_Y(m)}<\frac{\Theta\epsilon}{4}.
\end{align}

\noindent Then, by (\ref{3.2}), (\ref{a10}),
(\ref{3.8}), and (\ref{a11}), we get
\begin{align*}
\Big\Vert\big(\Theta (q-2N+n-1)&+n+1\big)dT_f(r)-\sum_{j=1}^q\omega(j)N_f(r,D_j)\notag\\
&\leq
\big((n+1)d+\frac{\Theta\epsilon}{2}\big)T_f(r)-\frac{n+1}{mH_Y(m)}N_{W(F)}(r).
\end{align*}
Therefore,
\begin{align}\label{3.9}
\Big\Vert\Theta d(q-2N+n-1-\frac{\epsilon}{2})T_f(r)\leq\sum_{j=1}^q\omega(j)N_f(r,D_j)-\frac{n+1}{mH_Y(m)}N_{W(F)}(r)
\end{align}

For each $J:=\{j_1,\dots,j_{H_Y(m)}\}\in \mathcal L,$ then there
exists a constant $\gamma_J \in \C, \gamma_J \ne0$ such that
\begin{align*}
W(F)=\gamma_J\cdot W(Q_1^{I_{j_11}}(f)\cdots
  Q_q^{I_{j_1q}}(f),\dots,Q_1^{I_{j_{H_Y(m)}1}}(f)\cdots
  Q_q^{I_{j_{H_Y(m)}q}}(f)).
\end{align*}
On the other hand, by (\ref{new2}) and Lemma \ref{L2.2},
$$\nu_{\frac{Q_1^{I_{j_11}}(f)\cdots
  Q_q^{I_{j_1q}}(f)\cdots Q_1^{I_{j_{H_Y(m)}1}}(f)\cdots
  Q_q^{I_{j_{H_Y(m)}q}}(f)}{W\big(Q_1^{I_{j_11}}(f)\cdots
  Q_q^{I_{j_1q}}(f),\dots,Q_1^{I_{j_{H_Y(m)}1}}(f)\cdots
  Q_q^{I_{j_{H_Y(m)}q}}(f)\big)}}\leq\sum_{1\leq i\leq
  H_Y(m)}\nu_{Q_1^{I_{j_i1}}(f)\cdots
  Q_q^{I_{j_iq}}(f)}^{[H_Y(m)-1]}.$$
Hence, for all $J\in \mathcal L,$ we have
\begin{align}\label{3.10}
\nu_{W(F)}&\geq\nu_{Q_1^{I_{j_11}}(f)\cdots
  Q_q^{I_{j_1q}}(f)\cdots Q_1^{I_{j_{H_Y(m)}1}}(f)\cdots
  Q_q^{I_{j_{H_Y(m)}q}}(f)}\notag\\
  &\quad\quad-\sum_{1\leq i\leq
  H_Y(m)}\nu_{Q_1^{I_{j_i1}}(f)\cdots
  Q_q^{I_{j_iq}}(f)}^{[H_Y(m)-1]}\notag\\
  &\geq\sum_{1\leq j\leq q}\sum_{i\in
  J}I_{ij}\big(\nu_{Q_j(f)}-\nu_{Q_j(f)}^{[H_Y(m)-1]}\big).
\end{align}

 For every $z\in \C,$ let $c_z:=(c_{1,z},\dots,c_{q,z})$
where $c_{j,z}:=\nu_{Q_j(f)}(z)-\nu_{Q_j(f)}^{[H_Y(m)-1]}(z).$
 Then, by definition of the Hilbert weight, there exists $J_z\in\mathcal
L$ such that
\begin{align*}
S_{Y}(m,c_z)=\sum_{i\in J_z} I_i\cdot c_z=\sum_{1\leq j\leq
q}\sum_{i\in
  J_z}I_{ij}\big(\nu_{Q_j(f)}(z)-\nu_{Q_j(f)}^{[H_Y(m)-1]}(z)\big).
\end{align*}
Then, by (\ref{new1}) and Lemma \ref{L2.1}, for every $K\in\mathcal
K$ we have
\begin{align*}
\frac{1}{mH_Y(m)}\sum_{1\leq j\leq q}\sum_{i\in
  J_z}&I_{ij}\big(\nu_{Q_j(f)}(z)-\nu_{Q_j(f)}^{[H_Y(m)-1]}(z)\big)\notag\\
  &\geq \frac{1}{n+1}\sum_{j\in
K}\big(\nu_{Q_j(f)}(z)-\nu_{Q_j(f)}^{[H_Y(m)-1]}(z)\big)\notag\\
&\quad\quad-\frac{(2n+1)\bigtriangleup}{m}\max_{1\leq
  j\leq q}
  \big(\nu_{Q_j(f)}(z)-\nu_{Q_j(f)}^{[H_Y(m)-1]}(z)\big)\notag\\
  &\geq \frac{1}{n+1}\sum_{j\in
K}\big(\nu_{Q_j(f)}(z)-\nu_{Q_j(f)}^{[H_Y(m)-1]}(z)\big)\notag\\
&\quad\quad-\frac{(2n+1)\bigtriangleup}{m}\sum_{1\leq
  j\leq q}
  \nu_{Q_j(f)}(z).
\end{align*}
Combining with (\ref{3.10}), for every $K\in\mathcal K$ and
$z\in\C,$ we have
\begin{align*}
\frac{1}{mH_Y(m)}\nu_{W(F)(z)} \geq \frac{1}{n+1}\sum_{j\in
K}\big(\nu_{Q_j(f)}(z)&-\nu_{Q_j(f)}^{[H_Y(m)-1]}(z)\big)\notag\\
&-\frac{(2n+1)\bigtriangleup}{m}\sum_{1\leq
  j\leq q}
  \nu_{Q_j(f)}(z).
\end{align*}
 This implies that
\begin{align}\label{a7}
\frac{n+1}{mH_Y(m)}\nu_{W(F)}\geq\max_{K\in
\mathcal K}\sum_{j\in K}\big(\nu_{Q_j(f)}&-\nu_{Q_j(f)}^{[H_Y(m)-1]}\big)\notag\\
&-\frac{(n+1)(2n+1)\bigtriangleup}{m}\sum_{1\leq
  j\leq q}
  \nu_{Q_j(f)}.
\end{align}

By Theorem \ref{Th25}, for any $z\in\C$ and any  $J\in\mathcal N,$  there exists subset $K'(J,z)\in\mathcal K,$ such that
\begin{align*}
\sum_{j\in J}\omega(j)\big(\nu_{Q_j(f(z))}-\nu_{Q_j(f(z))}^{[H_Y(m)-1]}\big)&\leq \sum_{j\in K'(J,z)}\big(\nu_{Q_j(f(z))}-\nu_{Q_j(f(z))}^{[H_Y(m)-1]}\big)\notag\\
&\leq\max_{K\in\mathcal K}\sum_{j\in K}\big(\nu_{Q_j(f(z))}-\nu_{Q_j(f(z))}^{[H_Y(m)-1]}\big).
\end{align*}
This implies that
\begin{align}\label{a8}
\max_{J\in\mathcal N}\sum_{j\in J}\omega(j)\big(\nu_{Q_j(f)}-\nu_{Q_j(f)}^{[H_Y(m)-1]}\big)\leq\max_{K\in\mathcal K}\sum_{j\in K}\big(\nu_{Q_j(f)}-\nu_{Q_j(f)}^{[H_Y(m)-1]}\big).
\end{align}
On the other hand, since the hypersurfaces $D_j\;(j=1,\dots,q)$ are in $N$-subgeneral position in $V,$ we have that for any $z\in \C$ there are at least (q-N) indices $j$ of $\{1,\dots,q\}$ such that $\nu_{Q_j(f)}(z)=0.$ Thus, we have
\begin{align*}
\sum_{j=1}^q\omega(j)\big(\nu_{Q_j(f)}-\nu_{Q_j(f)}^{[H_Y(m)-1]}\big)=\max_{J\in\mathcal N}\sum_{j\in J}\omega(j)\big(\nu_{Q_j(f)}-\nu_{Q_j(f)}^{[H_Y(m)-1]}\big).
\end{align*}
Combining with  (\ref{a8}), we have
\begin{align*}
\sum_{j=1}^q\omega(j)\big(\nu_{Q_j(f)}-\nu_{Q_j(f)}^{[H_Y(m)-1]}\big)\leq\max_{K\in\mathcal K}\sum_{j\in K}\big(\nu_{Q_j(f)}-\nu_{Q_j(f)}^{[H_Y(m)-1]}\big).
\end{align*}
Therefore, by (\ref{a7}) we have
\begin{align*}
\frac{n+1}{mH_Y(m)}\nu_{W(F)}\geq\sum_{j=1}^q\omega(j)\big(\nu_{Q_j(f)}-\nu_{Q_j(f)}^{[H_Y(m)-1]}\big)
-\frac{(n+1)(2n+1)\bigtriangleup}{m}\sum_{1\leq
  j\leq q}
  \nu_{Q_j(f)}.
\end{align*}
So, by integrating and by Jensen's formula, we get
\begin{align*}
\frac{n+1}{mH_Y(m)}N_{W(F)}(r)&\geq\sum_{j=1}^q\omega(j)\big(N_f(r,D_j)-N_f^{[H_Y(m)-1]}(r,D_j)\big)\notag\\
&\quad\quad-\frac{(n+1)(2n+1)\bigtriangleup}{m}\sum_{1\leq
  j\leq q}N_f(r, D_j)\\
&\geq \sum_{j=1}^q\omega(j)\big(N_f(r,D_j)-N_f^{[H_Y(m)-1]}(r,D_j)\big)\\
&\quad\quad-\frac{(n+1)(2n+1)dq\bigtriangleup}{m}\sum_{1\leq
  j\leq q}T_f(r)-O(1)\\
  &\geq\sum_{j=1}^q\omega(j)\big(N_f(r,D_j)-N_f^{[H_Y(m)-1]}(r,D_j)\big)-\frac{\Theta\epsilon}{4}T_f(r).
\end{align*}
Combining with  (\ref{3.9}) we get
 \begin{align*}
\Big\Vert\Theta d(q-2N+n-1-\epsilon)T_f(r)\leq \sum_{j=1}^q\omega(j)N_f^{[H_Y(m)-1]}(r,D_j).
\end{align*}
On the other hand,  $\omega(j)\leq\Theta$ by 
Proposition \ref{P24} (i), therefore 
 \begin{align}\label{final}
\Big\Vert(q-2N+n-1-\epsilon)T_f(r)\leq \sum_{j=1}^q
\frac{1}{d}N_f^{[H_Y(m)-1]}(r,D_j).
\end{align}
This completes the proof of Theorem \ref{Th4} and Proposition \ref{P5} in the special case of $\deg Q_j =d$ by the fact that 
$H_Y(m)-1\leq \big(
\begin{array}{c} q+m-1\\m \end{array}
\big) -1$, note that $Y \subset \C P^{q-1}$.\\

We now prove the theorem for the general case: $\deg Q_j=d_j.$
Denote by  $d$ the least common multiple of $d_1,\dots, d_q$ and put
$d_j^*:=\frac{d}{d_j}.$ By (\ref{final}) with the 
hypersurfaces $Q_j^{d_j^*}$ $(j\in\{1\dots,q\})$ of common degree
$d,$ we have
\begin{align*}
\Vert (q-2N+n-1-\varepsilon) T_f(r) &\leq \sum_{j=1}^q \frac{1}{d}
N^{[H_Y(m)-1]}_f(r,Q_j^{d_j^*})\notag\\
&\leq\sum_{j=1}^q \frac{d_j^*}{d}
N^{[[\frac{H_Y(m)-1}{d_j^*}+1]]}_f(r,Q_j)\notag \\
&\leq \sum_{j=1}^q \frac{1}{d_j} N^{[L_j]}_f(r,Q_j),
\end{align*}
where $$L_j:=[\frac{d_j(H_Y(m)-1)}{d}+1]
\leq \big[\frac{d_j \big(
\begin{array}{c}q+m-1\\m \end{array}
\big)}{d} +1
\big]
.$$ This completes the proof of  
Theorem \ref{Th4} and of Proposition \ref{P5}.
\hfill  $\square $\\

\vspace{1cm}

\newpage
\noindent  {\tac Gerd Dethloff}$^{1,2}$\\
 $^1$ Universit\'e Europ\'eenne de Bretagne, France\\
 $^2$
Universit\'{e} de Brest \\
   Laboratoire de math\'{e}matiques \\
UMR CNRS 6205\\
6, avenue Le Gorgeu, BP 452 \\
   29275 Brest Cedex, France \\
e-mail: gerd.dethloff@univ-brest.fr\\

\vspace{0.2cm}
\noindent {\tac Tran Van Tan}\\
Department of Mathematics\\
  Hanoi National University of Education\\
 136-Xuan Thuy street, Cau Giay, Hanoi, Vietnam\\
e-mail: tranvantanhn@yahoo.com; vtran@ictp.it\\

\vspace{0.2cm}
\noindent {\tac Do Duc Thai}\\
Department of Mathematics\\
  Hanoi National University of Education\\
 136-Xuan Thuy street, Cau Giay, Hanoi, Vietnam\\
e-mail: ducthai.do@gmail.com

\end{document}